%% file: MinorationFacile.tex
\documentclass[english,12pt]{amsart}
\input{Begin0}
\input{AN}

\usepackage[authoryear]{natbib}
\usepackage{hyperref}  		     

\input{Core}

\begin{document}
\title{About the lower bounds for the multiple testing problem}
\author{Yannick BARAUD}
\address{Universit\'e C\^{o}te d'Azur, CNRS, LJAD, France}
\email{baraud@unice.fr}
\date{\today}
\maketitle
\begin{abstract} 
Given an observed random variable $\bsX$, consider the problem of recovering its distribution among a finite family of candidate ones. The two-point inequality, Fano's lemma and more recently an inequality due to Venkataramanan and Johnson~\citeyearpar{MR3780042} allow to bound the maximal probability of error over the family from below.  The aim of this paper is to give a very short and simple proof of all these results simultaneously and improve in passing the inequality of Venkataramanan and Johnson.
\end{abstract}

\section{Introduction}
Consider a finite family $\sbP=\{\gP_{0}=\gp_{0}\cdot\gmu,\ldots,\gP_{N}=\gp_{N}\cdot\gmu\}$ of $N+1\ge 2$ distinct probabilities dominated by a measure $\gmu$ on a measurable space $(\gE,\cbE)$ and an observation $\bsX$ drawn from an unknown  probability $\gP\et$ that belongs to $\sbP$. A very basic but important problem in Statistics is to recover $\gP\et$ among $\sbP$ with a probability of error that we wish as small as possible. Given an estimator $\widehat T(\bsX)$ with values $\cJ=\{0,\ldots,N\}$, we evaluate its maximal risk by the quantity
\[
\sup_{j\in \cJ}\P_{j}\cro{\widehat T(\bsX)\neq j}=1-\inf_{j\in \cJ}\P_{j}\cro{\widehat T(\bsX)= j}=1-\overline \cR(\widehat T,\sbP)
\]
and the minimax risk associated to the family $\sbP$ is merely
\[
\inf_{\widehat T(\bsX)}\sup_{j\in \cJ}\P_{j}\cro{\widehat T(\bsX)\neq j}=1-\sup_{\widehat T(\bsX)}\inf_{j\in \cJ}\P_{j}\cro{\widehat T(\bsX)= j}=1-\overline \cR(\sbP)
\]
where $\inf_{\widehat T(\bsX)}$ and $\sup_{\widehat T(\bsX)}$ refer respectively to the infimum and the supremum over all estimators $\widehat T(\bsX)$ with values in $\cJ$. The problem of calculating the minimax risk is a difficult task in general and in most situations only upper and lower bounds can be established. In the present paper, we focus on the problem of bounding the minimax risk  from below or equivalently that of bounding the quantity $\overline \cR(\sbP)$ from above (which is more convenient). To tackle this problem, the basic approach is to replace the infimum by a mean: for all estimators $\widehat T(\bsX)$,
\[
\overline \cR(\widehat T,\sbP)=\inf_{j\in \cJ}\P_{j}\cro{\widehat T(\bsX)=j}\le \frac{1}{N+1}\sum_{j=0}^{N}\P_{j}\cro{\widehat T(\bsX)=j}=\overline \cB(\widehat T,\sbP).
\]
so that 
\begin{equation}\label{def-B}
\overline \cR(\sbP)\le \overline \cB(\sbP)=\sup_{\widehat T(\bsX)}\overline \cB(\widehat T,\sbP).
\end{equation}
It is therefore enough to bound $\overline \cB(\sbP)$ from above, what we shall do.

When $N=1$, the problem reduces to that of testing between the two simple hypotheses $``\gP\et=\gP_{0}"$ and  $``\gP\et=\gP_{1}"$. Denoting by $\norm{\gP_{0}-\gP_{1}}$ the total variation distance between $\gP_{0}$ and $\gP_{1}$, that is
\[
\norm{\gP_{0}-\gP_{1}}=\frac{1}{2}\int_{\gE}\ab{\gp_{1}-\gp_{0}}d\gmu=\int_{\gE}\pa{\gp_{1}-\gp_{0}}_{+}d\gmu
\]
where $(x)_{+}=\max\{x,0\}$ denotes the positive part of $x$, it is well known from Le Cam~\citeyearpar{MR0334381} that 
\begin{align}
2(1-\overline \cB(\sbP))&=\inf_{\widehat T(\bsX)}\cro{\P_{0}\cro{\widehat T(\bsX)=1}+\P_{1}\cro{\widehat T(\bsX)=0}}=1-\norm{\gP_{0}-\gP_{1}}\label{eq-2pts0}
\end{align}
which implies
\begin{equation}\label{eq-2pts}
\overline \cB(\sbP)\le\frac{1}{2}\pa{1+\norm{\gP_{0}-\gP_{1}}}.
\end{equation}
Inequality~\eref{eq-2pts} (which is actually an equality) shows that the problem of discriminating correctly between the probabilities $\gP_{0}$ and $\gP_{1}$ is impossible when their total variation distance is too small. It is also well-known that the infimum in \eref{eq-2pts0} is reached for a maximum likelihood estimator, that is any random variable $\widehat T\et(\bsX)$ that satisfies 
\[
\widehat T\et(\bsX)=
\begin{cases}
1 & \quad \text{if}\quad \gp_{1}(\bsX)>\gp_{0}(\bsX)\\
0 & \quad \text{if}\quad \gp_{1}(\bsX)<\gp_{0}(\bsX).
\end{cases}
\]

For $N\ge 2$, an alternative and celebrated bound is known as Fano's Lemma. It is based on the Kullback-Leibler divergence (KL-divergence for short). We recall that the KL-divergence between the probabilities $\gP=\gp\cdot\gmu$ and $\gQ=\gq\cdot \gmu$ is defined as
\[
K(\gP,\gQ)=
\begin{cases}
\int_{\gE}\gp\log\pa{\gp/\gq}d\gmu& \text{when}\quad \gP\ll\gQ\\
+\infty  & \text{otherwise}.
\end{cases}
\] 
One the one hand, the KL-divergence has the drawback to be possibly infinite (while the total variance distance remains always bounded by 1), on the other hand it possesses good properties with respect to product measures (which is the common situation when one observes a sample $\bsX=(X_{1},\ldots,X_{n})$): if $\gP=P^{\otimes n}$ and $\gQ=P^{\otimes n}$
\[
K(\gP,\gQ)=nK(P,Q).
\]
Unfortunately there is no such connection for the total variation distance which makes it unattractive. 

The following version of Fano's Lemma can be found in Ibragimov and Has'minskii~\citeyearpar{MR620321} page 325
\begin{equation}\label{eq-Fano}
\overline \cR(\sbP)\le \frac{\widetilde K+\log 2}{\log N}\quad \text{with} \quad  \widetilde K=\frac{1}{N+1}\sum_{j=0}^{N}K\pa{\gP_{j},\overline \gP},\quad \overline \gP=\frac{1}{N+1}\sum_{j=0}^{N}\gP_{j}.
\end{equation}
The proof is not very long but requires to know some Information Theory. It is however expected that the quantity $\log N$ in the denominator could be replaced by $\log(N+1)$. This of course makes no difference for large values of $N$ but leads to a slight improvement when these values are moderate. A step in this direction is due to Lucien Birg\'e~\citeyearpar{MR2241522} who suggested this new version of Fano's lemma:
\begin{equation}\label{eq-Lucien}
\overline \cR(\sbP)\le \kappa\vee \pa{\frac{(1+1/N)\overline K}{\log (N+1)}}\quad \text{with}\quad \overline K=\frac{1}{N+1}\sum_{j=0}^{N}K(\gP_{j},\gP_{0})
\end{equation}
and $\kappa=0.7$. Pascal Massart~\citeyearpar{MR2319879} proved an analogue of Lucien Birg\'e's result with a (worse) constant $\kappa\approx 0.84$ but a (slightly) shorter proof. Both proofs are moderately long and use the variational formula of the KL-divergence.

The main advantage of \eref{eq-Lucien} lies in the fact that $\log N$ is now replaced by $\log(N+1)$ but it requires that a probability among $\{\gP_{0},\ldots,\gP_{N}\}$, named $\gP_{0}$ here, dominates all the others. 

Very recently, Venkataramanan and Johnson~\citeyearpar{MR3780042} established a new inequality that provides an alternative to Fano's Lemma. Their result is the following: for all $\lambda>0$ and all probability $\gQ=\gq\cdot\gmu$ that dominates $\sbP$,
\begin{equation}\label{eq-VJ}
\overline \cR(\sbP)\le\overline \cB(\sbP)\le \frac{C(\lambda)}{N+1}\cro{\sum_{j=0}^{N}\int_{\gE}\gp_{j}^{1+\lambda}\gq^{-\lambda}d\gmu}^{1/(1+\lambda)},
\end{equation}
with $C(\lambda)={(1+\lambda)\lambda^{-\lambda/(1+\lambda)}}\ge 1$. The proof is quite short and requires very few tools. A nice feature of their bound lies in the facts that it is possible to choose the probability $\gQ$ to make the computation as simple as possible and also the value $\lambda>0$ to optimize the bound. We refer the reader to their paper for further details on the optimality of \eref{eq-VJ}, in particular how it compares and improves Fano's lemma, and how it contextualizes in the literature on multiple testing. 

In the next section we present a very simple approach that leads to a proof of \eref{eq-2pts}, an analogue of \eref{eq-Fano} and a (slight) improvement of \eref{eq-VJ}, namely $\eref{eq-VJ}$ with $C(\lambda)=1$.

\section{Our main result}
The proof of the following result is postponed to the end of the section.
\begin{thm}\label{thm-main}
For any probability $\gQ=\gq\cdot \gmu$ that dominates $\sbP$, 
\begin{equation}\label{thm-eq1}
1\le (N+1)\overline \cB(\sbP)=\E_{\gQ}\cro{\max_{j\in\cJ}\frac{\gp_{j}}{\gq}(\bsX)}
\end{equation}
where $\E_{\gQ}$ denotes the expectation with repect to $\gQ$. Moreover, for any function $\varphi$ that is convex, non-increasing and non-negative on $\R_{+}$,
\begin{equation}\label{thm-eq2}
\varphi\pa{(N+1)\overline \cB(\sbP)}\le \sum_{j=0}^{N}\E_{\gQ}\cro{\varphi\pa{\frac{\gp_{j}}{\gq}(\bsX)}}.
\end{equation}
\end{thm}

When $N=1$, choose $\varphi(u)=(u-1)_{+}$ and $\gQ=(\gP_{0}+\gP_{1})/2$ that fulfil the assumptions of our theorem. Inequalities~\eref{thm-eq1} and \eref{thm-eq2} immediately leads to 
\[
2\overline \cB(\sbP)-1\le \norm{\gP_{0}-\gQ}+ \norm{\gP_{1}-\gQ}= \norm{\gP_{1}-\gP_{0}}
\]
and gives~\eref{eq-2pts}.

For $N\ge 2$, choose 
\[
\varphi=\overline\varphi\1_{[1,+\infty)}\quad \text{with}\quad \overline\varphi(u)=u\log u-u+1\quad \text{for all $u\ge 0$.}
\]
It is easy to check that $\overline\varphi$ is non-negative, convex and satisfies $\varphi(1)=0$, hence $\varphi$ is 
convex, non-decreasing and non-negative as required. Since $\varphi\le \overline \varphi$, 
\[
\E_{\gQ}\cro{\varphi\pa{\frac{\gp_{j}}{\gq}(\bsX)}}\le \E_{\gQ}\cro{\overline \varphi\pa{\frac{\gp_{j}}{\gq}(\bsX)}}=K(\gP_{j},\gQ)\quad \text{for all $j\in\cJ$}
\]
and any $\gQ$ that dominates $\sbP$. Using~\eref{thm-eq1} we derive that
\begin{align*}
\varphi\pa{(N+1)\overline \cB(\sbP)}&=\overline \varphi\pa{(N+1)\overline \cB(\sbP)}\\
&=(N+1)\cro{\overline \cB(\sbP)\log(N+1)+\overline \varphi\pa{\overline \cB(\sbP)}}-N\\
&\ge (N+1)\overline \cB(\sbP)\log(N+1)-N.
\end{align*}
Taking $\gQ=\overline \gP$, we deduce from Inequality~\eref{thm-eq2} this analogue of Fano's Lemma \eref{eq-Fano}:
\[
\overline \cB(\sbP)\le \frac{\widetilde K+N/(N+1)}{\log(N+1)}\le \frac{\widetilde K+1}{\log(N+1)}.
\]

Finally taking for $u,\lambda\ge 0$, $\varphi(u)=u^{1+\lambda}$ we deduce from~\eref{thm-eq2} that
\[
\pa{(N+1)\overline \cB(\sbP)}^{1+\lambda}\le \sum_{j=1}^{N}\int_{\gE}\gp_{j}^{1+\lambda}\gq^{-\lambda}d\gmu,
\]
for any probability $\gQ=\gq\cdot \gmu$ that dominates $\sbP$, which implies~\eref{eq-VJ} with $C(\lambda)=1$.

\begin{proof}[Proof of Theorem~\ref{thm-main}]
Obviously $\gp_{j}\le \max_{j'\in\cJ} \gp_{j'}$ for all $j\in\cJ$. Hence, for all estimators $\widehat T(\bsX)$ with values in $\cJ=\{0,\ldots,N\}$, 
\begin{align*}
\overline \cB(\widehat T,\sbP)&=\frac{1}{N+1}\sum_{j=0}^{N}\gP_{j}\cro{\widehat T(\bsX)=j}= \frac{1}{N+1}\sum_{j=0}^{N}\int_{\gE}\1_{\{\widehat T(\gx)=j\}}\gp_{j}(\gx)d\gmu(\gx)\\
&\le \frac{1}{N+1}\sum_{j=0}^{N}\int_{\gE}\1_{\{\widehat T(\gx)=j\}}\max_{j\in\cJ}\gp_{j}(\gx) d\gmu(\gx)=\frac{1}{N+1}\int_{\gE}\max_{j\in\cJ}\gp_{j}(\gx) d\gmu(\gx).
\end{align*}
This upper bound is sharp and achieved for any maximum likelihood estimator, namely any estimator $\widehat T$ that satisfies 
\[
\gp_{\widehat T\et(\gx)}(\gx)=\max_{j\in\cJ}\gp_{j}(\gx)\quad \text{for $\gmu$-almost all $\gx\in\gE$.}
\]
Consequently, for any probability $\gQ=\gq\cdot \gmu$ that dominates $\sbP$,
\[
(N+1)\overline \cB(\sbP)=\int_{\gE}\max_{j\in\cJ}\gp_{j}(\gx) d\gmu(\gx)=\int_{\gE}\max_{j\in\cJ}\frac{\gp_{j}}{\gq}(\gx) \gq(\gx)d\gmu(\gx)=\E_{\gQ}\cro{\max_{j\in\cJ}\frac{\gp_{j}}{\gq}(\bsX)}
\]
which proves \eref{thm-eq1}. Finally, the properties of the function $\varphi$ (together with Jensen's inequality), lead to the following series of inequalities 
\begin{align*}
\varphi\pa{(N+1)\overline \cB(\sbP)}&= \varphi\pa{\E_{\gQ}\cro{\max_{j\in\cJ}\frac{\gp_{j}}{\gq}(\bsX) }}\le \E_{\gQ}\cro{\varphi\pa{\max_{j\in\cJ}\frac{\gp_{j}}{\gq}(\bsX) }}\\
&=\E_{\gQ}\cro{\max_{j\in\cJ}\varphi\pa{\frac{\gp_{j}}{\gq}(\bsX) }}\le \sum_{j=0}^{N+1}\E_{\gQ}\cro{\varphi\pa{\frac{\gp_{j}}{\gq}(\bsX) }},
\end{align*}
and finally to~\eref{thm-eq2}.
\end{proof}

\bibliographystyle{apalike}

\end{document}

%% file: Begin0.tex
\usepackage[utf8]{inputenc} 	
\usepackage[T1]{fontenc}
\usepackage{lmodern}
\usepackage{yhmath}

\usepackage{amsthm,amsmath,amssymb,amsbsy,bbm,mathrsfs,supertabular,
eurosym,graphicx,enumitem,xcolor}
\usepackage{mathtools}

\newtheoremstyle{BBstyle0}  {}{}{\itshape}{}{\bfseries}{}{6pt}{}
\newtheoremstyle{BBstyle1}  {3pt}{3pt}{\rmfamily}{}{\itshape}{: }{3pt}{}
\newtheoremstyle{BBstyle2}  {3pt}{3pt}{\itshape}{}{\bfseries\large}{}{0pt}{}
\newtheoremstyle{BBstyle3}  {}{}{\itshape}{}{\bfseries}{: }{3pt}{}
\newtheoremstyle{BBstyle4}  {}{}{\rmfamily}{}{\bfseries}{}{6pt}{}
  
\usepackage[authoryear]{natbib}

%% file: AN.tex
\newtheorem{thm}{Theorem}

\theoremstyle{definition}

\usepackage[english]{babel}

%% file: Core.tex
\newcommand{\pa}[1]{\left({#1}\right)}
\newcommand{\norm}[1]{\left\|{#1}\right\|}
\newcommand{\cro}[1]{\left[{#1}\right]}
\newcommand{\ab}[1]{\left|{#1}\right|}

\newcommand{\E}{{\mathbb{E}}}

\renewcommand{\P}{{\mathbb{P}}}
 
\newcommand{\R}{{\mathbb{R}}}

\DeclareMathAlphabet{\mathscrbf}{OMS}{mdugm}{b}{n}

\newcommand{\sbP}{{\mathscrbf{P}}}

\newcommand{\cB}{{\mathcal{B}}}

\newcommand{\cJ}{{\mathcal{J}}}

\newcommand{\cR}{{\mathcal{R}}}

\newcommand{\cbE}{\boldsymbol{\mathcal{E}}}

\newcommand{\gp}{{\mathbf{p}}}
\newcommand{\gq}{{\mathbf{q}}}

\newcommand{\gx}{{\mathbf{x}}}

\newcommand{\gE}{{\mathbf{E}}}

\newcommand{\gP}{{\mathbf{P}}}
\newcommand{\gQ}{{\mathbf{Q}}}

\newcommand{\bs}[1]{\boldsymbol{#1}}

\newcommand{\bsX}{{\bs{X}}}

\newcommand{\gmu}{{\bs{\mu}}}

\newlist{lista}{enumerate}{1}
\setlist[lista,1]{label=\alph*),ref=\alph*)}

\newlist{listi}{enumerate}{1}
\setlist[listi,1]{label=(\roman*),ref=(\roman*),align=left}

\newcommand{\eref}[1]{(\ref{#1})}
\renewcommand{\ge}{\geqslant}
\renewcommand{\le}{\leqslant}
\newcommand{\1}{1\hskip-2.6pt{\rm l}}

\newcommand{\et}{^{\star}}
